\documentclass[reqno]{amsart}
\usepackage{amsmath}
\usepackage{amsfonts}
\usepackage{a4wide}
\usepackage{amssymb,colortbl}
\usepackage{amsthm}
\theoremstyle{plain}
\begingroup
\newtheorem*{theorem*}{Theorem}
\newtheorem{theorem}{Theorem}
\newtheorem{lemma}[theorem]{Lemma}
\newtheorem{proposition}[theorem]{Proposition}
\newtheorem{corollary}[theorem]{Corollary}

\newtheorem{rmk}[theorem]{Remark}

\newtheorem{question}[theorem]{Question}
\endgroup

\theoremstyle{remark}
\begingroup
\endgroup

\mathchardef\emptyset="001F

\numberwithin{equation}{section}

\newcommand{\R}{{\mathbb R}}
\newcommand{\be}{\begin{equation}}
\newcommand{\ee}{\end{equation}}

\begin{document}
\title[Classification of uniformly distributed measures of dimension $1$]{Classification of uniformly distributed measures of dimension $1$ in general codimension}

\author{Paul Laurain}
\address{IMJ-Paris 7, Institut de Math\'ematiques de Jussieu, Equipe G\'eom\'etrie et Dynamique, B\^atiment Sophie Germain, Case 7012, 75205, Paris Cedex 13, France}
\email{paul.laurain@imj-prg.fr}

\author{Mircea Petrache}
\address{PUC Chile, Facultad de Matem\'aticas, Av. Vicuna Mackenna 4860, 6904441, Santiago, 
Chile}
\email{mpetrache@mat.uc.cl}
\date{\today}
 \maketitle

\begin{abstract}
Starting with the work of Preiss on the geometry of measures, the classification of uniform measures in $\mathbb R^d$ has remained open, except for $d=1$ and for compactly supported measures in $d=2$, and for codimension $1$. In this paper we study $1$-dimensional measures in $\mathbb R^d$ for all $d$ and classify uniform measures with connected $1$-dimensional support, which turn out to be homogeneous measures. We provide as well a partial classification of general uniform measures of dimension $1$ in the absence of the connected support hypothesis.
\end{abstract}

\section{Introduction}
Let $(X,d)$ be a metric space, and $\mu$ a Radon measure over $X$.
\begin{itemize}
\item We say that $\mu$ is a \emph{homogeneous measure} if there exists a subgroup $G$ of isometries of $(X,d)$ such that $\mathrm{spt}(\mu)$ is $G$-invariant, $G$ acts transitively on $\mathrm{spt}(\mu)$ and $I_\#\mu=\mu$ for all $I\in G$.
\item We say that $\mu$ is \emph{uniformly distributed} over $X$ if there exists a function $f:[0,+\infty)\to [0,+\infty)$ such that 
\begin{equation}\label{defn}\mbox{For every $x\in \mathrm{spt}(\mu)$ and all $r>0$ there holds $\mu(B(x,r))=f(r)$.}
\end{equation}
 \item We say that $\mu$ is \emph{uniformly distributed up to distance $r_0>0$} if there exists $f:[0,r_0]\to[0,+\infty)$ such that the above holds with the restriction $0<r\le r_0$. If such $r_0>0$ exists then we say that $\mu$ is \emph{locally uniformly distributed}.
\end{itemize} 
One can check that the above notions are expressed in increasing order of generality.

\medskip

Note that each of the above-mentioned classes of measures is invariant under the action of dilations and translations and products, but in general \emph{not under sums}. For example, for any $a,b\in X$ the measure $\delta_a+\delta_b$ is homogeneous, however e.g. for $(X,d)$ equal to $\mathbb R$ with the usual distance, the measure $\mu=\delta_0+\delta_1+\delta_2+\delta_3$ is not locally uniform up to distance $r_0>0$ if $r_0\ge 3$. 

\medskip

The present paper is a step in the investigation of the second notion in the case of $X=\mathbb R^n$, and to introduce the known results we focus on the following question:

\begin{question}\label{question1}
 What form could the function $f(r)$ in \eqref{defn} take?
\end{question}

\noindent\textbf{Assumption:} \emph{We concentrate in this paper on the case that $(X,d)$ is the Euclidean space $(\mathbb R^d, d_{\ell_2})$, and we will not mention this hypothesis below.}

\medskip

The question \ref{question1} appears as a natural development of \cite{preiss, marstrand, mattila}, to whom the following is due:

\begin{theorem*}[Marstrand, Mattila, Preiss]
 If $\mu$ is uniformly distributed then there exist $c\in\mathbb R\setminus \{0\}$ and an integer $k\le n$ such that $f(r)=cr^k + o(r^k)$.
\end{theorem*}

In fact, Preiss proves the above result for uniformly distributed measures as the main step towards the more general, by now classical, result, which instead of local uniformity assumes ``uniformity up to the first germ of $f$'':

\begin{theorem*}[Preiss]
 If for an integer $k$ there holds $\mu(B(x,r))= cr^k + o(r^k)$ for $\mu$-almost every $x$ with $c$ independent of $x$, then $\mu$ is a constant multiple the Hausdorff measure $\mathcal H^k$ restricted to a countable union of smooth $k$-dimensional manifolds.
\end{theorem*}

A self-contained explanation of the above theorems is given in the book \cite{delellis}.

Following the breakthrough in \cite{preiss}, Question \ref{question1} in Euclidean spaces has been directly addressed in the related work \cite{kop}, in which the case $f(r)=cr^{n-1}$ is studied and the first example of a \emph{non-homogeneous uniform measure} is given, and \cite{kip}, which represents what can be said about Question \ref{question1} in the case of general $f$ by general methods.

\begin{theorem}[Kirchheim-Preiss]\label{kpthm}
 If $\mu$ is a uniformly distributed measure over $\mathbb R^d$ then 
 \begin{equation}\label{fanalitic}
  f(r) \mbox{ is analytic in an interval $[0,r_0)$,}
 \end{equation}
and in fact there exists an integer $0\le k\le d$ and an analytic variety $V\subset \mathbb R^d$ such that $\mu=C\mathcal H_k\llcorner V$. If $V$ is compat then it is an algebraic variety.
\end{theorem}
As a consequence of the study in Kowalski-Preiss it is clear that even at the germ level, $V$ is much more rigid than stated above, but a geometric characterization is at the moment missing. We only know the complete classification for the case that $f$ coincides with the distribution function $f$ of a subspace of codimension $1$:
\begin{theorem*}[Kowalski-Preiss]
 If in Theorem \ref{kpthm} we have $k=n-1>1$ and $f(r)=cr^{n-1}$, then $V$ from Theorem \ref{kpthm} can be, up to isometry, $\mathbb R^{n-1}$, or $\mathbb R^{n-m-1}\times \mathbb S^m_r$ where $\mathbb S^m_r$ is a $m$-dimensional sphere of radius $r$, or $\mathbb R^{n-4}\times \mathcal C^3$, where $\mathcal C^3:=\{(x_1,x_2,x_3,x_4)\in\mathbb R^4: x_1^2=x_2^2+x_3^2+x_4^2\}$. 
\end{theorem*}
A recent result by Nimer \cite{nimer} produces further examples of non-homogeneous uniform measures in codimension higher than $1$ (obtained as a union of tangent spheres in $4$ dimensions and the cone over this set in $5$ dimensions) and a regularity result \cite{nimerreg} saying that singular set of $V$ is of codimension at least $3$ inside $V$, which in particular implies that all $2$-dimensional uniformly distributed measures are nonsingular.

\medskip

\begin{rmk}[locally uniformly distributed $\mu$] Note that a large amount of work has been done to understand the properties of $f(r)$ in the case of uniformly distributed measures. Very little is known for the case of a locally uniformly distributed measure.
\end{rmk}

\medskip

In \cite{kip} the uniformly distributed measures over $\mathbb R$ were completely classified, as were the uniformly distributed measures of bounded support over $\mathbb R^2$. Restricted to these two classes, ``there was no surprise'' in the sense that the only uniformly distributed measures in these two classes are the homogeneous measures.

\begin{theorem}[Kirchheim-Preiss]\label{kpthm2}
A uniformly distributed measure over $\mathbb R$ is either a multiple of the Lebesgue measure, equals up to affine transformations a multiple of the counting measure on $\mathbb Z$, or a multiple of the counting measure on $\mathbb Z +\{0,b\}$, where $b\notin \mathbb Z$.

\medskip

A uniformly distributed measure over $\mathbb R^2$ with bounded support is either a multiple of $\mathcal H^1$ restricted to a circle, or a multiple of the counting measure on the vertices of a regular $n$-gon, or a multiple of the counting measure of the union of vertices of two regular $n$-gons which are inscribed in a common circle.
\end{theorem}

We point out that the below questions (stated in increasing order of difficulty) are still open:

\begin{question}[$0$-dimensional uniformly distributed measures]\label{quest0dim}\hfill
\begin{enumerate}
\item Are there any $0$-dimensional uniformly distributed measures which are not homogeneous? \item Are there any $0$-dimensional uniformly distributed measures over $\mathbb R^2$ which are not homogeneous? 
\item Can one classify the $0$-dimensional uniformly distributed measures in $\mathbb R^d$ for $d\ge 2$?
\end{enumerate}
\end{question}

In this paper we consider the case of measures which are the Hausdorff $1$-dimensional measure restricted to a curve $\gamma$ in $\mathbb R^d$. We want to calculate the \emph{local} constraints imposed from the requirement that $\mathcal H^1\llcorner \gamma$ is a uniformly distributed measure. In other words, we require that there exists $f:[0,+\infty)\to[0,+\infty)$ and $r_0>0$ such that for all $r<r_0$ and for all $x\in\gamma$ there holds
\begin{equation}\label{gammaintersball}
\mathcal H^1(\gamma\cap B(x,r)) = f(r).
\end{equation}
It turns out that this translates into a constraint on the curvatures of $\gamma$, a fact which allows us to prove the following improvement of Theorem \ref{kpthm} for the case that $\mu$ has $1$-dimensional connected support. 
\begin{theorem}\label{mainthm}
Let $\mu:=\mathcal H^1\llcorner \gamma$, where $\gamma$ is a curve in $\mathbb R^d$. The following are equivalent:
\begin{enumerate} 
\item $\mu$ is a uniformly distributed measure. 
\item $\gamma$ is a toric knot or a generalized helix, in particular all the curvatures of $\gamma$ are constant.
\item $\mu$ is a homogeneous measure.
\end{enumerate}
\end{theorem}
In order to explain what is meant in the above second point, we recall this classical differential geometry topic here, see 2.16 of \cite{Ku} for details.

\begin{proposition}[Classification of constant curvature curves]\label{propclassifcurv}
 If $\gamma\subset \mathbb R^d$ is a curve all of whose curvatures are constant, and assuming that $\mathrm{AffSpan}(\gamma)=\mathbb R^d$ (or, equivalently, that all curvatures of $\gamma$ are nonzero), then up to translation and rotation, $\gamma$ can be parametrized as follows
 \begin{itemize}
  \item$\gamma(t)=\left(r_1\cos(\alpha_1t),r_1\sin(\alpha_1t),\dots, r_{\frac{d}{2}} \cos\left(\alpha_{\frac{d}{2}}t\right),r_{\frac{d}{2}},\sin\left(\alpha_{\frac{d}{2}}t\right)\right)$ if $d$ is even, 
  \item $\gamma(t)=\left(r_1\cos(\alpha_1t),r_1\sin(\alpha_1t),\dots, r_{\frac{d-1}{2}} \cos\left(\alpha_{\frac{d-1}{2}}t\right),r_{\frac{d-1}{2}}\sin\left(\alpha_{\frac{d-1}{2}}t\right), bt\right) $if $d$ is odd,
\end{itemize}
where the rotation speeds $\alpha_1,\ldots,\alpha_{\left\lfloor\frac{d}{2}\right\rfloor}$ are nonzero and the radiii $r_1,\ldots,r_{\left\lfloor\frac{d}{2}\right\rfloor}$ as well as the translation speed $b$ are positive.
\end{proposition}

The curves as in the first part of the above proposition are called \emph{toric curves} and the curves as in the second part are called \emph{helices}. A closed (or equvalently, periodic) toric curve is called a \emph{toric knot}.

\begin{rmk}
Note that all helices give rise to uniform measures $\mathcal H^1\llcorner \gamma$, whereas in the toric curve case $\mathcal H^1\llcorner\gamma$ is locally finite only if the curve $\gamma$ is actually a toric knot.
\end{rmk}
%\begin{corollary}
%Any  one-dimensional uniform measure $\mu$ of $\R^d$ consists of an at most countable disjoint union of parallel helices?? 
%\end{corollary}
We can extend the statement of Theorem \ref{mainthm} to a necessare condition the non-connected case with following result, which can still yield partial classifications in combination with Theorem \ref{kpthm2}.

\begin{theorem}\label{mainthm2}
Let $\mu$ be a uniformly distributed measure of dimension $1$ in $\mathbb R^d$. Then there exist 
\begin{enumerate}
\item an integer $1\le k\le d$ and a curve $\gamma_0 \subset \mathbb R^k\times\{0\}^{d-k}\subset \mathbb R^d$ with constant curvatures $\kappa_j$ which are zero if and only if $j\ge k$, 
\item a discrete set $X\subset \mathrm{Isom}(\mathbb R^d)$,
\end{enumerate}
such that $\mu$ is, up to isometry of $\mathbb R^d$, equal to
\[
c\mathcal H^1\llcorner \left(\bigcup_{R\in X}R\gamma_0\right).
\]
Furthermore, each pair of the above isometric copies of $\gamma_0$ are at constant positive distance from each other, and if $\gamma_0$ is a helix then all the copies $R\gamma_0$ have axes parallel to the one of $\gamma_0$.
\end{theorem}
We are tempted to conjecture that the above set $X$ has further restrictions to its structure, like in the case $d=3$, see proposition below, but we leave the study of the general form of this result to future work.
\begin{proposition}\label{3d}
Let $\mu$ be a uniformly distributed measure of dimension $1$ in $\mathbb R^3$. Then $\mu$ has one of the following forms:
\begin{enumerate}
 \item Either $\mu$ is the product of a $0$-dimensional uniform measure in $H$ with a $1$-dimensional uniform measure in $H^\perp$, where $H\subset \mathbb R^2$ is an affine subspace of dimension $1$ or $2$
 \item Or up to isometry $\mu=c\mathcal H^1\llcorner(X+\gamma)$, where $\gamma(t)=(r_1\sin(\alpha t), r_1\cos(\alpha t), bt)$ and $X=\{0\}^2\times\left(\frac{2\pi b}{n}\mathbb Z\cup \left(a+\frac{2\pi b}{n}\mathbb Z\right)\right)$ with $n\in\mathbb N$ and $a\in\mathbb R$. 
\end{enumerate}
In the latter case $\mu$ is a homogeneous measure.
\end{proposition}

Let us remark that in the first case a complete classification rely on an answer to question \ref{quest0dim} with $d=2$. We expect that in higher dimension a similar result, that is to say: the actions of the discrete group of isometries $X$ reduces to each plane of rotation to one or two regular-polygons, see proposition  11 of \cite{kip}, to a $1$-dimensional measure along the axis of the helix as described in point (2) above and finally to a $0$-dimensional measure in the orthogonal of the span affine of the generating curve.

\medskip

\textbf{Acknowledgements:} The authors acknowledge the Institut Henri Poincar\`e for a ``Research in Paris'' grant in July 2018, which allowed essential progress on this paper. MP was supported by the \emph{Fondecyt Iniciaci\'on} grant 11170264 ``Sharp Asymptotics for Large Particle Systems and Topological Singularities''. MP wishes to thank Dali Nimer for enlightening discussions on related topics and support from the Max Planck Institute for Mathematics in Bonn in 2017, which allowed these discussions to take place.

\medskip

\section{Proof of Theorem \ref{mainthm}}
\subsection{Taylor polynomial expansion and preliminaries}
We initially follow the natural setup already utilized in \cite{kop}, which only uses the assumption that $\gamma$ is $C^k$ for $k$ big enough (here we assume $k\geq 3$). We will use the Taylor polynomial approximations of $\gamma, \gamma'$ which we denote as follows:
\begin{equation}
\label{gamma}
\gamma(x)=A_1x + A_2x^2 + \cdots + A_kx^k +R(x)
\end{equation}
where $A_j\in\mathbb R^n, j=1,\ldots,k$ and $\bar A_{j-1}=\frac{A_j}{j}$, and $R(x)=O(x^{k+1}), \bar R(x)=O(x^k)$. Then we have
\begin{eqnarray}
|\gamma(x)|^2 &=& \sum_{2\le j<k+2}\left(\sum_{h+h'=j, h,h'\ge1}\langle A_h, A_{h'}\rangle\right)x^j + O(x^{k+2})\nonumber\\
&=&\sum_{2\le j<k+2}C_jx^j + O(x^{k+2}),\label{coeffc}
\end{eqnarray}
We will parameterize $\gamma$ by arclength, and thus 
\begin{equation}\label{normal1}
 |\gamma'|\equiv 1.
\end{equation}
Due to this normalization, in \eqref{coeffc} we have $C_2=\langle A_1,A_1\rangle=1$, thus
\begin{equation}\label{eqgammasquare}
 |\gamma(x)|^2= x^2 +\sum_{3\le j<k+2}C_j x^j + O(x^{k+2}).
\end{equation}
In order to test condition \eqref{gammaintersball} we have to integrate $|\gamma'|$ over the region 
\begin{equation}\label{ballintersgamma}
 \gamma\cap B(r,0)=\{x:\ |\gamma(x)|^2 \le r^2\},
\end{equation}
and as we normalized $|\gamma'|\equiv 1$ the function $f$ from \eqref{gammaintersball} is given by the equation
\begin{equation}\label{definf}
 f(r)=f_+(r)-f_-(r),\quad \mbox{where }f_\pm(r)\mbox{ are the two solutions of }\quad |\gamma(f(r))|^2=r^2.
\end{equation}
We will profit of our normalizations in order to solve \eqref{definf}, and what we use is the formula for the formal solution of $G\circ F=r^2$ where $G,F$ are power series in $r$, with $G$ being the right hand side of \eqref{eqgammasquare} and $F$ being the taylor polynomial for $f$, which we assume to also be $C^k$-regular, so that
\begin{equation}\label{eqtaylorf}
 f(s)=\sum_{1\le j<k+2} c_js^j + O(s^{k+2}).
\end{equation}
The formal solution $F$ for the power series equation $G(F(r))=r^2$ has terms up to order $k$ which coincide with the $k$-th Taylor polynomial of $f$.

\medskip

By explicit computation, from \eqref{eqgammasquare} and with the notation \eqref{eqtaylorf}, the first terms in the expansion of $|\gamma(f(s))|$ around the value $s=0$ read as follows 

\be
\begin{split}
|\gamma(f(s))|^2 &=c_1^2s^2+(2c_1c_2 + c_1^3C_3)s^3+(c_2^2+2c_3c_1+ 3c_1^2c_2C_3+c_1^4C_4)s^4+O(s^5)%\\
% &+(2c_1c_4+2c_2c_3 + 3(c_1^2c_3+c_1c_2^2)C_3+ 4c_1^3c_2C_4 + c_1^5C_5) s^5 \\
% &+O(s^6)
\end{split}
\ee

Hence for the two solutions $f_+>0>f_-$ of the functional equation $|\gamma(f(s))|^2=s^2$ we require that $c_1=\pm 1$ in order to match the coefficients of $s^2$, and the values of the $c_j,j>1$ are uniquely determined, giving 
\be\label{formula_fpm}
\begin{split}
f_\pm(s)&=\pm s-\frac{C_3}{2}s^2 \mp \frac{1}{2} \left(C_4 -5\frac{C_3^2}{4}\right)s^3+ O(s^4)
\end{split}
\ee
and thus
\begin{equation}\label{findfr}
f(r)= f_+(r) -f_-(r)=2r-\left(C_4-5\frac{C_3^2}{2}\right)r^3+O(r^4).
\end{equation}
As $\mathcal H_1\llcorner\gamma$ is a uniformly distributed measure, in particular this expression does not depend on the choice of an origin along $\gamma$, and the coefficients $C_4$ appearing above need to be constant. The coefficient $C_3$ seems to be free but in fact it is vanishing, indeed since the parametrization is done by arclenght, we have $C_3=2\langle \gamma'(0),\gamma''(0)\rangle=0$. The generalization of this formula will be part of our proof in Section \ref{proofsprop} below.

\medskip

Our aim is to show that we may first choose special coordinates so that some of the coefficients disappear, and so that the remaining coefficients have a simple geometric interpretation. We thus completely classify the possible $\gamma$ as being exactly the constant-curvature curves. We prove the following:

\begin{proposition}\label{mainthmpart1}
 Assume that $d\ge 1$ and $\gamma$ is a $C^{2d}$-regular curve parameterized by arclength, $\gamma:\mathbb R\to\mathbb R^d$ and let $\mu:=\mathcal H^1\llcorner \gamma$. If $[a,b]\subset \mathbb R$ is an interval and we have
 \[
\sup_{p,q\in[a,b]}\lim_{r\downarrow 0}\frac{\mu(B(\gamma(p),r))-\mu(B(\gamma(q),r))}{r^{2d}}=0,
 \]
then $\gamma|_{[a,b]}$ has constant curvatures $\kappa_1,\ldots,\kappa_{d-1}$.
\end{proposition}

\begin{corollary}\label{mainthmpart1cor}
If $\gamma$ is a curve in $\mathbb R^d, d\ge 2$, and $\mathcal H^1\llcorner \gamma$ is a uniform measure, then $\gamma$ has constant curvatures $\kappa_1,\ldots,\kappa_{d-1}$. 
\end{corollary}
This shows that $(1)\Rightarrow (2)$ in Theorem \ref{mainthm}. The implication $(3)\Rightarrow (1)$ is easy to obtain and is valid in general for measures of any metric space, as already mentioned in the introduction: Indeed, homogeneity implies that $\mu(B(x,r))=\mu(I(B(x,r)))$ for any isometry $I\in G$, that for any $x,y\in \mathrm{spt}(\mu)$ there exists $I\in G$ such that $I(x)=y$, and if $I$ is an isometry this implies that $I(B(x,r))=B(y,r)$. This implies that $\mu$ is uniformly distributed. The implication $(2)\Rightarrow (3)$ is a direct consequence of the following (more general) proposition:

\begin{proposition}\label{mainthmpart2}
 A finite $1$-dimensional measure $\mu$ is homogeneous if and only if it is a constant multiple of $\mathcal H^1\llcorner (G\cdot \gamma)$ where $G\cdot\gamma$ is the orbit of $\gamma$ under a discrete subgroup $G<\mathrm{Isom}(\mathbb R^d)$ and $\gamma$ is a curve all of whose curvatures are constant.
\end{proposition}

The above result is well-known but we were not able to find a proof in the literature, thus we produce a proof below. Also, we will not discuss here the classification of possible discrete groups $G$ described above.

\medskip

In order to recall the definition of the $\kappa_j$, we first work in the absence of degeneracies of $\gamma$. We say that $\gamma$ has a \emph{degenerate point} at $x$ in case we have the following zero-Wronskian condition:
\[
\gamma'(x)\wedge\gamma''(x)\wedge\cdots\wedge\gamma^{(n)}(x)=0. 
\]
If $x$ is a nondegenerate point then we define the \emph{Frenet frame} $(E_1(x),\ldots, E_n(x))$ at $\gamma(x)$ inductively applying the Gram-Schmidt orthogonalization procedure to the basis $(\gamma'(x),\ldots,\gamma^{(n)}(x))$.

\medskip

If $x$ is a nondegenerate point of $\gamma$ then the curvatures $\kappa_j(x), 1\le j<n$ can be directly defined by
\begin{equation}\label{defkj}
 \kappa_j(x):=\langle E_j'(x), E_{j+1}(x)\rangle.
\end{equation}
In this case the following Frenet-Serret equations follow directly from \eqref{defkj} and hold in a neighborhood of $x$:
\begin{equation}\label{frenetframe}
\frac{d}{ds}\left(\begin{array}{c}E_1 \\ \vdots \\ \vdots \\ E_n \end{array}\right)=  \left(\begin{array}{cccc}0 & \kappa_1 &  & 0 \\-\kappa_1 & \ddots & \ddots  &  \\ &    \ddots & 0 & \kappa_{n-1} \\  0 &  &-\kappa_{n-1}&0 \end{array}\right)\left(\begin{array}{c}E_1 \\ \vdots \\ \vdots  \\ E_n \end{array}\right).
\end{equation}
Note that we can also define the curvature when the point $x$ is degenerate (e.g. for a straight line in $\R^3$) but we have to be more careful with the definition. In case there exists a frame $(E_1(x),\ldots, E_n(x))$ at $\gamma(x)$ for each $x$ in the domain of definition of $x$ and there exist curvatures $\kappa_1(x),\ldots,\kappa_{n-1}(x)$ such that \eqref{frenetframe} holds, then  call $\gamma$ is called a \emph{Frenet curve}. Nomizu \cite{nomizu} proved the following (his paper is set up in $\mathbb R^3$ for a finite length curve, but the proof generalizes directly to possibly infinite curves and to general dimension):
\begin{theorem}[Nomizu 1959]\label{nomizuthm}
 Assume that $\gamma$ is a $C^\infty$ curve such that at each $x$ there exists $m_x\in\mathbb N_+$ such that $\gamma^{(m_x)}(x)\neq 0$ (such curves are called by Nomizu \emph{normal curves}). Then $\gamma$ is a Frenet curve.
\end{theorem}
The main nontrivial observation in the proof \cite{nomizu} of the above result is that the normality condition allows, for smooth curves, to deduce that degenerate points are isolated. After this, one defines the $E_j$ piecewise uniquely along $\gamma$ by Gram-Schmidt orthogonalization, and verifies that these definitions agree across the isolated degenerate points, due to the regularity of $\gamma$. We use the following:

\begin{corollary}[of Theorem \ref{nomizuthm}]\label{cor1}
 If $\gamma$ is an analytic curve then it is Frenet.
\end{corollary}
This is useful to us in conjunction with the Kirchheim-Preiss theorem \ref{kpthm}, which gives a sufficient condition for uniformly distributed measures already mentioned in the introduction.

\medskip

Note that an analytic variety can only self-intersect at an at most countable number of points. By comparing the asymptotics for $r\to 0$ of $\mathcal H^1\llcorner \gamma (B(r,x))$ when $x$ is a self-intersection point and when it is not, we conclude that $\gamma$ has no self-intersections, and thus we get the following:
\begin{corollary}[of Theorem \ref{kpthm}]\label{cor2}
 If $\gamma$ is a curve and $\mathcal H_1\llcorner \gamma$ is uniformly distributed then $\gamma$ is analytic and embedded.
\end{corollary}
Summing up all the preliminary results so far, we get:
\begin{corollary}[of corollaries \ref{cor1} and \ref{cor2}]\label{cor3}
 If $\gamma$ is a curve and $\mathcal H^1\llcorner \gamma$ is uniformly distributed, then an orthonormal frame satisfying \eqref{frenetframe} along $\gamma$ exists.
\end{corollary}

\subsection{Proofs of propositions \ref{mainthmpart1} and \ref{mainthmpart2}}\label{proofsprop}

\begin{proof}[Proof of Proposition \ref{mainthmpart1}:]\hfill\\

\textbf{Step 1:} \emph{Expressing the $c_k$ in terms of the $C_h$.}

\medskip

We consider the general form of the coefficients $c_i$ from \eqref{eqtaylorf}, which are determined by the uniform measure conditions \eqref{definf}. We saw above that $c_1^2=1$ is the condition on the coefficient of $r^2$, and this gives $c_1=\pm 1$. We prove now that the remaining coefficients of $f$ are determined uniquely in terms of $c_1$ and of the $C_j$, thus we define $f_\pm$ to be the solution determined by $c_1=\pm1$ respectively. The coefficient of $s^k$ for $k>2$ in the expression of $|\gamma(f(s))|^2$ for general $f$ can be computed explicitly, and it equals
\[
\sum_{2\le k'\le k} C_{k'}\left(\sum_{\substack{\vec \ell \in \{1,2,\ldots\}^{k'}\\\ell_1+\cdots+\ell_{k'}=k}} c_{\ell_1}\cdots c_{\ell_{k'}}\right).
\]
The highest index $\ell_\alpha$ which can appear in the above expression is $\ell_\alpha=k-1$ and it can appear only for $k'=2$, with a contribution of $2c_1c_{k-1}C_2=2c_1c_{k-1}$. The above sum also contains a single term that features a $C_k$-factor, and this term is $c_1^kC_k$. Thus for $k>2$, imposing that the coefficient of $s^k$ is zero and isolating the $c_{k-1}$-contribution in the ensuing equation gives:
\begin{equation}\label{eqck}
c_{k-1}=\frac{1}{2c_1}\left[P + c_1^k C_k\right], \quad P\in \mathbb N[C_2,C_3,\ldots,C_{k-1},\ c_1,c_2,\ldots,c_{k-2}].
\end{equation}
This allows to prove by induction on $k$ that for $k\ge 2$ the coefficient $c_k$ is uniquely determined as polynomial of $C_2,C_3,\ldots,C_{k+1}$ with coefficients depending on $c_1=\pm1$:
\begin{equation}\label{c_kpol}
 c_k=P_{\pm,k}(C_2,\ldots,C_{k}) +\frac{(\pm1)^{k-1}}{2} C_{k+1}.
\end{equation}
In particular, $f_\pm$ are uniquely determined (this proves in general that a formula generalizing \eqref{formula_fpm} for $f_\pm$ holds). 

\medskip

\textbf{Step 2:} \emph{The $c_{k-1}$ with $k$ even appear in the Taylor expansion of $f(r)$ without cancellations.}

\medskip

By symmetry under the transformation $s\mapsto -s$ we see that $f_-(s)=f_+(s)$. Thus in the calculation of $f_+-f_-$ the coefficients $c_{k-1}$ with $k$ even do not cancel, and thus stay constant as we vary the choice of origin along $\gamma$.

\medskip 

\textbf{Step 3:} \emph{Frenet frame coordinates and constancy of $\kappa_1$.}

\medskip

We know from Corollary \ref{cor3} that $\gamma$ has a Frenet frame $E_1(s),\ldots,E_n(s)$ along $\gamma$ as in \eqref{frenetframe}. In order to continue the proof, we will use this frame coordinates to express, first the $\gamma^{(k)}(0)$, and then we carefully discuss which contributions appear in the formula for $C_k$.

\medskip

We fix coordinates in $\mathbb R^n$ such that for $s=0$ there holds 
\begin{equation}\label{coord0}
 (E_1(0),\ldots,E_n(0))=(e_1,\ldots,e_n).
\end{equation}
Expansion \eqref{gamma} gives the following identifications of the first coefficients:
\begin{equation}\label{rewritegamma}
 \gamma'(0)=E_1(0)=A_1=e_1,\quad \gamma''(0)=2A_2=|\gamma''(0)|E_2(0)=2|A_2|e_2. 
\end{equation}
In particular we have $C_3=2\langle A_1,A_2\rangle=0$ and \eqref{defkj} gives also
\begin{equation}\label{kappa1}\kappa_1(0)=\langle E_1'(0), E_2(0)\rangle=4|A_2|^2.
\end{equation}
Differentiating twice the formula $|\gamma'(x)|^2=1$ we obtain also $\langle \gamma'(x),\gamma'''(x)\rangle = -|\gamma''(x)|^2$, thus $\langle \gamma'(0),\gamma'''(0)\rangle=-4|A_2|^2$, and we can compute
\[
 C_4=|A_2|^2+2\langle A_1,A_3\rangle =|A_2|^2 + 2\left\langle \gamma'(0),\frac{\gamma'''(0)}{6}\right\rangle =|A_2|^2-\frac{4}{3}|A_2|^2 \stackrel{\mbox{\eqref{kappa1}}}{=}-\frac{1}{12}\kappa_1(0).
\]
The constancy of the coefficients of $f(r)$ (computed up to $O(s^5)$ in \eqref{findfr}) as we vary the choice of origin along $\gamma$ imply that $\kappa_1$ is constant on $\gamma$.

\medskip

If $\kappa_1\equiv 0$ then $\gamma$ is a straight line, and the thesis is proved. So we further consider only the case $\kappa_1\equiv c\neq 0$. 

\medskip

\textbf{Step 4:} \emph{Induction step for the constancy of $\kappa_h, h\ge 2$ and strategy of proof.}

\medskip

We now prove by induction that the remaining $\kappa_h$, for $h< n$, are constant as well. The case $h=1$ is proved. For the inductive step, we may assume that $\kappa_j$ are constant and nonzero for $1\le j\le h$, and we prove that $\kappa_{h+1}$ is constant too. We note that, as above, if $\kappa_{h+1}\equiv 0$ then $\gamma$ is contained in an affine subspace of dimension $h+1$ and $\kappa_{h'}\equiv 0$ for all $h'>h$ as well, and the proof is complete. Similarly, if $h+1=n$ then the proof is complete. If neither of the last two mentioned conditions are satisfied, then we can continue the induction.

\medskip

The proof of the inductive step will proceed as follows: first we discuss in a general setup what terms can and cannot appear in the expression of $C_k$; next, we show that, under the condition that $\kappa_h$ and lower curvatures are constant, $C_{2h+3}=0$; next, we easily show, using the results of steps 1 and 2, that $C_{2(h+2)}$ needs to be constant; finally, from the constancy of $C_{2(h+2)}$ entails the constancy of $\kappa_{h+1}$ and allows to conclude the induction.

\medskip

\textbf{Step 5:} \emph{Encoding Frenet frame expansion of derivative via paths on a graph.}

\medskip

More generally, since $\alpha! A_\alpha=\gamma^{(\alpha)}(0)$ by Taylor expansion of $\gamma$ and comparison to \eqref{gamma}, we find that $C_k$ is a linear combination of terms of the form $\langle \gamma^{(\alpha)}(0), \gamma^{(\beta)}(0)\rangle$, with $\alpha+\beta=k$. We can then express $\gamma^{(\alpha)}(0)$ in the basis \eqref{coord0}, and use the equations \eqref{frenetframe} for bookkeeping which terms may and may not occur. We can imagine the following graph
\begin{equation}\label{graph}
 *\longrightarrow E_1\stackrel{\kappa_1}{\longrightarrow}E_2\stackrel{\kappa_2}{\longrightarrow}E_3\stackrel{\kappa_3}{\longrightarrow}\cdots\stackrel{\kappa_{n-2}}{\longrightarrow}E_{n-1}\stackrel{\kappa_{n-1}}{\longrightarrow}E_n,
\end{equation}
and for fixed $\alpha\geq 1$ we consider paths which
\begin{enumerate}
 \item start at ``$*$'' and at the first step move to the position labelled ``$E_1$'',
 \item at successive steps move between neighbors only, by exactly one position, and do not visit ``$*$'' ever again,
 \item stop after a number $\alpha'\le \alpha$ of steps.
\end{enumerate}
Now suppose we have such a path $P$, whose final position is ``$E_j$''. We now assign a monomial in the $\kappa_h$'s and their derivatives to $P$ as follows:
\begin{enumerate}
\item if $\alpha'=1$ then $P=1$,
 \item each time $P$ follows an arrow labelled $\kappa_h$ from the left to the right, we multiply by $\kappa_h$ our monomial,
 \item each time $P$ follows an arrow labelled $\kappa_h$ from the right to the left, we multiply by $-\kappa_h$ our monomial,
 \item if $P$ had length $\alpha'<\alpha$ then we choose (with repetitions allowed) $\alpha-\alpha'$ amongst the $\alpha'-1$ factors $\kappa_h$ in our constructed monomial (again considered including repetitions), and put a derivative on each of them.   
\end{enumerate}
To make this clearer, here is an example. Suppose $\alpha=11, \alpha'=7$. Then a possible $P$ is the path given by following first $*\longrightarrow E_1$ and then arrows $\kappa_1\kappa_2\kappa_3(-\kappa_3)(-\kappa_2)\kappa_2$, and terminating at position $E_3$. Then we must distribute $\alpha-\alpha'=4$ derivatives in total amongst these $6$ monomials in one of many possible ways, ending up, for example, with the resulting monomial $\kappa_1\kappa_2^{(3)}\kappa_2'\kappa_2\kappa_3^2$.

\medskip

One can verify by induction on $\alpha$ that with respect to the basis $E_1(x),\ldots,E_n(x)$, the $E_j(x)$-component of $\gamma^{(\alpha)}(x)$ has coefficient given by the sum of all possible monomials constructed as above associated to paths ending up at ``$E_j$''. To see this, note that we need to distribute the $\alpha$ derivatives by Leibnitz rule. When we apply a first derivative to $\gamma$ we get $E_1$, and each time one further derivative falls on the $E_k$-component of our expression, we can then apply \eqref{frenetframe} to replace that term one amongst the ones $E_{k\pm1}$ appearing in \eqref{graph} to the immediate left/right, adding the corresponding factor of $\pm \kappa_{k\pm1}$ to our expansion. The further details of the verifications are left to the reader.

\medskip

The orthogonality of the $E_j$'s and the fact that $C_k$ only features terms coming from scalar products $\langle \gamma^{(\alpha)}(0),\gamma^{(\beta)}(0)\rangle$ with $\alpha+\beta=k$, means that the paths $P$ associated to the $\gamma^{(\alpha)}$ term and $Q$ associated to the $\gamma^{(\beta)}$ term need to terminate on the same $E_j$, in order to have a chance of participating to $C_k$'s expansion. 

\medskip
 
\textbf{Step 6:} \emph{The coefficients $C_{2h+3}$ in \eqref{coeffc} vanish.}

\medskip

If we now consider the scalar products $\langle \gamma^{(\alpha)}(0),\gamma^{(\beta)}(0)\rangle$ which participate to $C_{2h+3}$, thus $\alpha+\beta=2h+3$. We suppose $\alpha\ge\beta$. By the final paragraph of Step 5, in $C_{2h+3}$ is a sum of contributions from paths $P_\alpha, P_\beta$ with associated monomials in the $\kappa_j$'s of degrees $\alpha'\le\alpha-1$ and $\beta'\le\beta-1$ respectively, arriving to some common endpoint $E_j$. Thus if we immagine to concatenate $P_\alpha$ and $P_\beta$ we obtain a loop inside the graph \eqref{graph}. If $h'$ is the highest index of $\kappa_{h'}$ appearing in such loop, we find
\[
 2h+1\ge\alpha'+\beta'\ge 2h'\quad \Rightarrow\quad h'\le h,
\]
which implies that for $h'\ge h+1$, no derivatives of $\kappa_{h'}$ can appear in either monomial. As for $h'\le h$ we have $\kappa_{h'}\equiv const$ by inductive assumption, it follows that no derivatives of any $\kappa_{h'}$ appear at all in the associated monomials, and we have $\alpha'=\alpha-1, \beta'=\beta-1$, thus the loop formed by concatenating $P_\alpha,P_\beta$ must have length $\alpha+\beta=2h+3$. However this cannoh be, as a loop has even length. This shows that all contributions to $C_{2h+3}$ vanish, so that $C_{2h+3}=0$, as claimed.

\medskip

\textbf{Step 7:} \emph{The coefficients $C_{2(h+2)}$ from \eqref{coeffc} are independent of the choice of origin along $\gamma$.}

\medskip 

We know that $C_2=1$ as a consequence of our normalization conditions and during the induction over $h$ we prove that $C_{2(h'+2)}$ is independent of the choice of origin for $h'<h$ and that $C_{2h'+3}=0$ for $h'\le h$ by Step 6. We also know that coefficient $c_{2h+3}$ appears in the Taylor expansion of $f(r)$ without canceling, due to Step 2, and to the fact that $2h+3$ is odd. In formula \eqref{c_kpol} with $k=2h+3$, we thus have that all the coefficients in the left hand side except the one in $C_{2(h+2)}$ are independent of the choice of origin along $\gamma$, and we have that $c_{2h+3}$ appears in the expression for $f(r)$ and thus is independent of the choice of origin as well. This shows that $C_{2(h+2)}$ is independent of the choice of origin as well, as claimed.

\medskip

\textbf{Step 8:} \emph{The curvature $\kappa_{h+1}$ is independent of the choice of origin along $\gamma$.}

\medskip

To see this, knowing that we know the statement to be true for $\kappa_{h'}, h'\le h$, and in view of Step 7, we know that $C_{2(h+2)}$ is also independent of the choice of origin, it will suffice to know that fixing the values of $\kappa_{h'},h'\le h$ and of $C_{2(h+2)}$ determines $\kappa_{h+1}$ completely. This will follow once we prove the following:

\medskip

\textbf{Claim:} \emph{We can express $C_{2(h+2)}$ as a polynomial in the $\kappa_{h'}, h'\le h+1$, and the only monomial featuring $\kappa_{h+1}$ in this polynomial is a positive multiple of $(\kappa_1\kappa_2\cdots\kappa_{h+1})^2$.}

\medskip

Indeed, once we prove the claim, knowing that $\kappa_{h'}, h'\le h$ are all with isolated zero, we get that $\kappa_{h+1}^2$ is uniquely determined in terms of $\kappa_{h'}, h'\le h$ and $C_{2(h+2)}$. Thus $\kappa_{h+1}$ is determined and constant as we move the origin of our coordinates along $\gamma$, up to sign. If $\kappa_{h+1}^2$ is zero then we are done, and if it is nonzero then, by regularity of $\gamma$, the value $\kappa_{h+1}$ cannot jump as we move the origin and thus the sign choice needs to be constant as well, completing the proof of the current step.

\medskip

Now we prove the above Claim. To do this we completely classify the contributions of each term $\langle \gamma^{(\alpha)}(0), \gamma^{(\beta)}(0)\rangle$ contributing to $C_{2(h+2)}$. Like we did in Step 7, we consider paths $P_\alpha,P_\beta$ contributing to the above two terms, assuming that $P_\alpha,P_\beta$ end up in a given $E_j$. Concatenating these paths gives a loop of length at most $2(h+2)$ and since we consider only terms containing $\kappa_{h+1}$-factors this loop must contain vertices ``$*$'' and ``$E_{h+2}$''. Therefore the loop is uniquely determined: it must be the only loop without backtracking, which goes from ``$*$'' to ``$E_{h+2}$'' and back. By following the assignment of signs in Step 5, we find that if paths $P_\alpha,P_\beta$ meet at $E_j$ and $\alpha\ge\beta$, then $P_\alpha$ is contributing the factors $\kappa_1\kappa_2\cdots\kappa_{h+1}(-\kappa_{h+1})(-\kappa_h)\cdots(-\kappa_j)$  and $P_\beta$ is contributing the factors $\kappa_1\kappa_2\cdots\kappa_{j-1}$. Therefore we get indeed the monomial $(\kappa_1\kappa_2\cdots\kappa_{h+1})^2$, with a sign of $(-1)^{h-j}=(-1)^{\alpha'-\beta'}=(-1)^{\alpha'+\beta'}=(-1)^{2h+2}=+1$. We thus show that all the coefficients appear with positive sign, completing the proof of the claim.

\medskip

As the induction step is complete, we have completed the proof of the proposition \ref{mainthmpart1}.
\end{proof} 

\begin{proof}[Proof of Proposition \ref{mainthmpart2}:]
Assume that $\mu$ is one-dimensional and homogeneous; in particular it is uniformly distributed, so that by the result of Kirchheim and Preiss, see Theorem \ref{kpthm}, it is $c\mathcal H^1$ restricted to an analytic subvariety $V$ of dimension $1$. We can assume that $c=1$ and that  $0\in V$ without loss of generality. By homogeneity, $V$ is the orbit of a subgroup $\tilde G\subset \mathrm{Isom}(\mathbb R^n)$ and $\tilde G$ itself has Hausdorff dimension $1$. As $V$ is analytic, we have that $\tilde G$'s connected components are at positive bounded from below Hausdorff distance from each other (the elementary proof of this is similar to the one of Lemma \ref{nonaccum} below, and we leave it to the reader), thus $\tilde G=G\times G_0$, where $G_0$ is the connected component of the identity in $\tilde G$ and $G$ is a discrete subgroup. Therefore we may reduce to the case where $G$ is trivial, the orbit of $\tilde G=G_0$ is a curve and that , and that $V=\gamma$ is an analytic curve. We find that by homogeneity the curvatures along $\gamma$ must be constant, concluding the proof of one implication. 

\medskip

Viceversa, assume that $\mu = c\mathcal H^1\llcorner G\cdot\gamma$ and $\gamma$ has constant curvatures. Without loss of generality $G=id$, otherwise we consider $\mu$ componentwise. Up to diminishing the dimension $n$ we may assume that all the curvatures of $\gamma$ are nonzero, and thus the Frenet frame $(E_1,\ldots,E_n)$ along $\gamma$ is nondegenerate. This means that there exists a unique curve of rotations $x\mapsto R(x)\in O(n)$ which sends $(e_1,\ldots,e_n)$ to $(E_1,\ldots,E_n)$ calculated at point $\gamma(x)$, and up to a change of basis we may suppose that $R(0)=id$. We can obtain the group property $R(x+y)=R(x)R(y)$ by uniqueness of the solution to the Frenet frame equation \eqref{frenetframe} with constant curvatures. Thus $R(x)$ is a one-parameter subgroup of $\mathrm{Isom}(\mathbb R^n)$.
\end{proof}

Let us recall from the classification stated in the introduction that curves which have constant curvatures all of which are different than zero have no axis if their affine span has even dimension, they have an axis if the affine span has odd dimension.

\section{Proof of Theorem \ref{mainthm2}}
We start with the following result. Recall that for two closed sets $A,B\subset X$ in a metric space (below $X=\mathbb R^d$ with the Euclidean distance), their distance is defined as
\[
 \mathrm{dist}(A,B):= \inf\{d(x,y):\ x\in A,y\in B\}.
\]
We start with the following consequence of the Kirchheim-Preiss classification from Theorem \ref{kpthm}, and of our Proposition \ref{mainthmpart1}:
\begin{lemma}\label{nonaccum}
 Assume tha $\mu$ is a uniformly distributed measure of dimension $1$. Then the connected components of $\mathrm{spt}(\mu)$ are congruent helices which lie at positive bounded from below distance from each other.
\end{lemma}
\begin{proof}
By Theorem \ref{kpthm} we may suppose $\mu =\mathcal H^1\llcorner V$ where $V$ is an analytic variety of dimension $1$. Since $V$ is countably $1$-rectifiable, locally near regular points it is an analytic curve, and by applying Proposition \ref{mainthmpart1} we find that each component has constant curvatures, and the values of the constants involved are the same for each component. As the curve is uniquely determined by its curvatures, and since connected components of $V$ do not accumulate at every point of $V$ (because $\mu$ is locally finite), we have that connected components of $V$ stay at positive distance from each other.
\end{proof}

We next show that in fact the connected components of $\mathrm{spt}(\mu)$ are translations of each other:

\begin{proof}[Proof of Theorem \ref{mainthm2}:]
From Lemma \ref{nonaccum} it follows that $V=\mathrm{spt}(\mu)$ has at most countably many connected components, which we denote by $\gamma_1,\ldots, \gamma_n,\ldots$. We now call $C:\gamma_1\to \bigcup_{j\ge 2}\gamma_j$ the multivalued map which associates to each $x\in\gamma_1$ the indices of other connected components at which the distance to the complement of $\gamma_1$ in $V$ is achieved:
\[
 C(x):=\left\{j\ge 2:\ \mathrm{dist}(x,V\setminus \gamma_1)=\mathrm{dist}(x,\gamma_j)\right\}.
\]
We denote the distance between $\gamma_1$ and its ``first neighbors'' amongst the connected components by $d_1$:
\[
 d_1:=\mathrm{dist}(\gamma_1,V\setminus\gamma_1).
\]
Note that $C(x)$ must have constant cardinality, as follows by noting that $\lim_{r\downarrow d_1}\mu(B(x,r))$ is independent of $x\in\gamma_1$. Then, assuming that this constant cardinality is at least $1$, it follows that the subsets of $\gamma_1$ given by 
\[
 S_j:=\{x\in\gamma_1:\ j\in C(x)\}
\]
form an at most countable cover of $\gamma_1$, thus at least one of the $S_j$ has accumulation points. Using the fact that the curves $\gamma_1,\gamma_j$ are analytic, we find that whenever $S_j$ has accumulation points then $S_j=\gamma_1$. We suppose for the sake of concreteness that this happens for $j=2$, i.e. that $S_2=\gamma_1$. 

\medskip

We claim that if $\gamma_j$ are helices then the axes of $\gamma_1$ and $\gamma_2$ are parallel. Indeed, were they not, it would imply that for $x\to\pm\infty$ we have $\mathrm{dist}(\gamma_2(x),\gamma_1)\to \infty$, contradicting the fact that $S_2=\gamma_1$. Let $d_a\ge 0$ be the distance between the axes of $\gamma_1$ and $\gamma_2$ if they are helices, and the distance between their centers if they are torus knots.

\medskip

If for some $j$ the set $S_j$ is composed of isolated points, then we have a contradiction, again by using the property that $\lim_{r\downarrow d_1}\mu(B(x,r))$ is independent of $x\in\gamma_1$, as above. Thus all the first neighbors $\gamma_j$ are translates of $\gamma_1$ by vectors orthogonal to its affine span or rotated copies within its associated torus. 

\medskip

We can next repeat the reasoning replacing the role of $\gamma_1$ by the subset of $\gamma_1$ and its first neighbors instead of $\gamma_1$, we find that also the next distance $d_2$ realized between $\gamma_1$ and the ``second neighbors'' is constant along each one of these second neighbors, and if they are helices, their axes are parallel too. The same reasoning can be iterated showing these properties for all the $\gamma_j$. This concludes the proof of the theorem.
\end{proof}

\begin{proof}[Proof of Proposition \ref{3d}:]
Assuming that we are not in the first case of the statement corresponds to saying that all the curvatures of a connected component $\gamma$ are nonzero, i.e. that connected components are helices, and as a consequence of Theorem \ref{mainthmpart2} we know that their axes are parallel. 

\medskip

We show first that the distance $d_a$ between axes of first-neighbor connected components $\gamma_1,\gamma_2$ cannot be positive. If by contradiction $d_a>0$, then the axis of $\gamma_2$ is a translation by $v\neq 0$ of the axis of $\gamma_2$, with $v$ perpendicular to both axes. Then let $x_-,x_+$ be points in $\gamma_1$ with lowest (resp. highest) $v$-coordinate. We see that $\mathrm{dist}(x_-,\gamma_2)\ge d_a$ and $\mathrm{dist}(x_+,\gamma_2)\ge d_a$ by comparing with the projection orthogonal to the axes. Thus we have that the distance between a point in $\gamma_1$ and the curve $\gamma_2$, which is constant by Theorem \ref{mainthmpart2}, must be equal to $d_a$, and $\gamma_2=v+\gamma_1$. At the same time, $v$ must be orthogonal to both curves at each point, which is a contradiction to the fact that their affine span is $\mathbb R^d$. Thus $d_a=0$ and $\gamma_1,\gamma_2$ have the same axis. The same reasoning can be applied to each ``next-neighbor'' connected component to show that all connected components of the support of $\mu$ are helices with the same axis.

\medskip

We next use the notation and coordinates from Proposition \ref{propclassifcurv} for the helix connected components of $\mathrm{spt}(\mu)$. As the connected components also have the same periods $2\pi b\mathbb Z$, we find that by quotienting by the group of periods we obtain a uniform measure on $\{x\in\mathbb R^2:\ |x|=r_1\}\times \mathbb Z/2\pi b\mathbb Z$, whose connected components are circles at constant distance, thus are parallel circles. Quotienting by the orthogonal of these circles we find a uniform measure on a circle, which by Theorem \ref{kpthm} is given by two translated (possibly coincident) copies of a regular $n$-gon. This concludes the proof.
\end{proof}

\bibliographystyle{plain}
\bibliography{measure}

\end{document}